\documentclass[11pt,reqno,sumlimits]{amsart}
\usepackage{amsfonts, amsmath, amscd, amssymb, euscript, amsthm, array, booktabs,
dcolumn, shortvrb, tabularx, units, url}
\usepackage[pdftex]{graphicx}
\usepackage[pdftex]{hyperref}
\usepackage[all]{xy}
\normalsize
\DeclareMathOperator{\rk}{rank}
\DeclareMathOperator{\topo}{top}
\DeclareMathOperator{\der}{deR}
\DeclareMathOperator{\todd}{Todd}
\DeclareMathOperator{\BS}{BS}
\DeclareMathOperator{\an}{an}
\DeclareMathOperator{\ind}{ind}
\DeclareMathOperator{\LL}{L}
\DeclareMathOperator{\ST}{ST}
\DeclareMathOperator{\FL}{FL}
\DeclareMathOperator{\SF}{SF}
\DeclareMathOperator{\exact}{exact}
\DeclareMathOperator{\SSSS}{SS}
\DeclareMathOperator{\st}{Struct}

\DeclareMathOperator{\spin}{Spin}

\DeclareMathOperator{\odd}{odd}
\DeclareMathOperator{\even}{even}
\DeclareMathOperator{\im}{Im}

\DeclareMathOperator{\ch}{ch}

\DeclareMathOperator{\U}{U}
\DeclareMathOperator{\BU}{BU}

\DeclareMathOperator{\tr}{tr}
\DeclareMathOperator{\CS}{CS}
\DeclareMathOperator{\cs}{cs}
\DeclareMathOperator{\dirlim}{\underrightarrow{\lim}}
\begin{document}
\setlength{\baselineskip}{1.4\baselineskip}
\newtheorem{defi}{Definition}
\newtheorem{lemma}{Lemma}
\newtheorem{remark}{Remark}
\newtheorem{coro}{Corollary}
\newtheorem{exam}{Example}
\newtheorem{thm}{Theorem}
\newtheorem{prop}{Proposition}
\newcommand{\wt}[1]{{\widetilde{#1}}}
\newcommand{\ov}[1]{{\overline{#1}}}
\newcommand{\wh}[1]{{\widehat{#1}}}
\newcommand{\poin}{Poincar$\acute{\textrm{e }}$}
\newcommand{\deff}[1]{{\bf\emph{#1}}}
\newcommand{\boo}[1]{\boldsymbol{#1}}
\newcommand{\abs}[1]{\lvert#1\rvert}
\newcommand{\norm}[1]{\lVert#1\rVert}
\newcommand{\inner}[1]{\langle#1\rangle}
\newcommand{\poisson}[1]{\{#1\}}
\newcommand{\biginner}[1]{\Big\langle#1\Big\rangle}
\newcommand{\set}[1]{\{#1\}}
\newcommand{\Bigset}[1]{\Big\{#1\Big\}}
\newcommand{\BBigset}[1]{\bigg\{#1\bigg\}}
\newcommand{\dis}[1]{$\displaystyle#1$}
\newcommand{\R}{\mathbb{R}}
\newcommand{\N}{\mathbb{N}}
\newcommand{\Z}{\mathbb{Z}}
\newcommand{\Q}{\mathbb{Q}}
\newcommand{\E}{\mathcal{E}}
\newcommand{\G}{\mathcal{G}}
\newcommand{\F}{\mathcal{F}}
\newcommand{\V}{\mathcal{V}}
\newcommand{\W}{\mathcal{W}}
\newcommand{\SSS}{\mathcal{S}}
\newcommand{\h}{\mathbb{H}}
\newcommand{\g}{\mathfrak{g}}
\newcommand{\C}{\mathbb{C}}
\newcommand{\A}{\mathcal{A}}
\newcommand{\M}{\mathcal{M}}
\newcommand{\HH}{\mathcal{H}}
\newcommand{\D}{\mathcal{D}}
\newcommand{\PP}{\mathcal{P}}
\newcommand{\K}{\mathcal{K}}
\newcommand{\RR}{\mathcal{R}}
\newcommand{\RRR}{\mathscr{R}}
\newcommand{\DDD}{\mathscr{D}}
\newcommand{\so}{\mathfrak{so}}
\newcommand{\gl}{\mathfrak{gl}}
\newcommand{\uu}{\mathfrak{u}}
\newcommand{\aaa}{\mathbb{A}}
\newcommand{\bbb}{\mathbb{B}}
\newcommand{\DD}{\mathsf{D}}
\newcommand{\ccc}{\bold{c}}
\newcommand{\sss}{\mathbb{S}}
\newcommand{\cdd}[1]{\[\begin{CD}#1\end{CD}\]}
\normalsize
\title[The Simons-Sullivan differential analytic index]{The differential analytic
index in Simons-Sullivan differential $K$-theory}
\author{Man-Ho Ho}
\address{Department of Mathematics and Statistics\\ Boston University}
\email{homanho@bu.edu}
\maketitle
\nocite{*}
\begin{center}
\emph{\small Dedicated to my father Kar-Ming Ho}
\end{center}
\begin{abstract}
We define the Simons-Sullivan differential analytic index by translating the
Freed-Lott differential analytic index via explicit ring isomorphisms between
Freed-Lott differential $K$-theory and Simons-Sullivan differential $K$-theory.
We prove the differential Grothendieck-Riemann-Roch theorem in Simons-Sullivan
differential $K$-theory using a theorem of Bismut.
\end{abstract}
\tableofcontents
\section{Introduction}
\noindent As explained in \cite{BS09}, \cite{BS10a}, \cite{DF00}, \cite{HS05}, the
physics motivation for differential $K$-theory is a quantum field theory whose
Largrangian has differential form field strength. This leads to a generalized
cohomology theory with a map to ordinary cohomology that implements charge quantization.
In \cite{DF00} Freed argued that there should be a similar extension of topological
$K$-theory. We refer to \cite[$\S 1.4$]{FL10} for a historical discussion. The
mathematical motivation for differential $K$-theory can be traced to Cheeger-Simons
differential characters \cite{CS85}, the unique differential extension of ordinary
cohomology \cite{SS08a}, and to the work of Karoubi \cite{K86}. It is thus natural to
look for differential extensions of generalized cohomology theories, for example
topological $K$-theory. The differential extension of topological $K$-theory is now
known as differential $K$-theory. Roughly speaking, differential $K$-theory is
topological $K$-theory combined with differential form data in a complicated way,
just as differential characters combine ordinary cohomology with differential form
data. Various definitions of differential $K$-theory have been proposed, notably by
Bunke-Schick \cite{BS09}, Freed-Lott \cite{FL10}, Hopkins-Singer \cite{HS05} and
Simons-Sullivan \cite{SS10}. Axioms for differential extensions of generalized
cohomology theories are given in \cite{BS10a}, where it is shown that two differential
extensions of a fixed generalized cohomology theory satisfying certain conditions are
uniquely isomorphic. In particular the four models of differential $K$-theory mentioned
above are isomorphic by this abstract result. For more details and an introduction to
differential $K$-theory, see \cite{BS10}.\\\\
The Atiyah-Singer family index theorem can be formulated as the equality of the analytic
and topological pushforward maps $\ind^{\an}=\ind^{\topo}:K(X)\to K(B)$. Applying the
Chern character, we get the Grothendieck-Riemann-Roch theorem, the commutativity of the
following diagram
\cdd{K(X) @>\ch>> H^{\even}(X; \Q) \\ @V\ind^{\an}VV @VV\int_{X/B}\todd(X/B)\cup(\cdot) V
\\ K(B) @>>\ch> H^{\even}(B; \Q)}
Analogous theorems hold in differential $K$-theory. Bunke-Schick prove the differential
Grothendieck-Riemann-Roch theorem (dGRR) \cite[Theorem 6.19]{BS09}, i.e., for a proper
submersion $\pi:X\to B$ of even relative dimension, the following diagram is commutative:
\cdd{\wh{K}_{\BS}(X) @>\wh{\ch}_{\BS}>> \wh{H}^{\even}(X; \R/\Q) \\
@V\ind^{\an}_{\BS}VV @VV\wh{\int_{X/B}}\wh{\todd}(\wh{\nabla}^{T^VX})\ast(\cdot)
V \\ \wh{K}_{\BS}(B) @>>\wh{\ch}_{\BS}> \wh{H}^{\even}(B; \R/\Q)}
where $\wh{H}(X; \R/\Q)$ is the ring of differential characters \cite{CS85},
$\wh{\ch}_{\BS}$ is the differential Chern character \cite[\S 6.2]{BS09}, $\ind^{\an}_{\BS}$
is the Bunke-Schick differential analytic index \cite[\S 3]{BS09} and \dis{\wh{\int_{X/B}}
\wh{\todd}(\wh{\nabla}^{T^VX})\ast} is a modified pushforward of differential characters
\cite[\S 6.4]{BS09}. The notation is explained more fully in later sections. Freed-Lott
prove the differential family index theorem \cite[Theorem 7.32]{FL10} $\ind^{\an}_{\FL}
=\ind^{\topo}_{\FL}:\wh{K}_{\FL}(X)\to\wh{K}_{\FL}(B)$, where $\ind^{\an}_{\FL}$ and
$\ind^{\topo}_{\FL}$ are the Freed-Lott differential analytic index
\cite[Definition 3.11]{FL10} and the differential topological index \cite[Definition 5.33]{FL10}.
Applying the differential Chern character $\wh{\ch}_{\FL}$ yields the dGRR
\cite[Corollary 8.23]{FL10}. Since $\ind^{\an}_{\BS}=\ind^{\an}_{\FL}$
\cite[Corollary 5.5]{BS09}, the two dGGR theorems are essentially the same. See
\cite{H11a} for a short proof of the dGRR.\\\\
To this point, the differential index theorem formulated in Simons-Sullivan
differential $K$-theory has not appeared. The purpose of this paper is to fill this gap
by both defining the differential analytic index and proving the dGRR in Simons-Sullivan
differential $K$-theory.\\\\
The first main result of this paper (Theorem \ref{thm 1}) is the construction of
explicit ring isomorphisms between Simons-Sullivan differential $K$-theory and
Freed-Lott differential $K$-theory. While these theories must be isomorphic by
\cite[Theorem 3.10]{BS10a}, the explicit isomorphisms have not been appeared in
literature as far as we know. Moreover, it follows from Corollary \ref{coro 1} that the
flat parts of these differential $K$-theories are also isomorphic via the restriction
of the explicit ring isomorphisms in Theorem \ref{thm 1}. This result is a more explicit
version of \cite[Theorem 5.5]{BS10a} in this case. The advantage of these explicit ring
isomorphisms is that we see which elements in these differential $K$-groups correspond
to each other.\\\\
The second main result of this paper is the dGRR in Simons-Sullivan differential
$K$-theory. We first define the Simons-Sullivan differential analytic index by
translating the Freed-Lott analytic index via the explicit isomorphisms in Theorem
\ref{thm 1}. To be precise, we study the special case where the family of kernels
$\ker(D^E)$ forms a superbundle. The general case follows from a standard perturbation
argument as in \cite[\S 7]{FL10}. The Simons-Sullivan differential analytic index
of an element $\E=[E, h^E, [\nabla^E]]\in\wh{K}_{\SSSS}(X)$, in the special case,
is given by
$$\ind^{\an}_{\SSSS}(\E)=[\ker(D^E), h^{\ker(D^E)}, [\nabla^{\ker(D^E)}]]
+[V, h^V, [\nabla^V]]-[\dim(V), h, [d]],$$
where $[V, h^V, [\nabla^V]]:=\wh{\CS}^{-1}(\wt{\eta}(\E))$, and all the terms will be
explained below. The general case of $\ind^{\an}_{\SSSS}(\E)$ is given by a similar
formula. This formula is considerably more complicated than the Freed-Lott differential
analytic index. This indicates that Simons-Sullivan differential $K$-theory is not the
easiest setting for differential index theory, although the Simons-Sullivan construction
of the differential $K$-group is perhaps the simplest among all the existing ones. We
then prove the dGRR (Theorem \ref{iaithm 3}) in the special case, i.e., the
commutativity of the following diagram
\cdd{\wh{K}_{\SSSS}(X) @>\wh{\ch}_{\SSSS}>> \wh{H}^{\even}(X; \R/\Q) \\
@V\ind^{\an}_{\SSSS}VV @VV\wh{\int_{X/B}}\wh{\todd}(\wh{\nabla}^{T^VX})\ast(\cdot) V \\
\wh{K}_{\SSSS}(B) @>>\wh{\ch}_{\SSSS}> \wh{K}_{\SSSS}(B)}
in Simons-Sullivan differential $K$-theory, using a theorem of Bismut
\cite[Theorem 1.15]{B05}. The general case of the dGRR follows by a similar argument,
since \cite[Theorem 1.15]{B05} can be extended to the general case.\\\\
In principle all the theorems and proofs can be transported from Freed-Lott differential
$K$-theory to Simons-Sullivan differential $K$-theory by the explicit isomorphisms given
by Theorem \ref{thm 1}. However, with \cite[Theorem 1.15]{B05} the proof of the dGRR is
easier.\\\\
The paper is organized as follows: the next two sections contain all the necessary
background material. Section 2 reviews Simons-Sullivan differential $K$-theory. Section
3 reviews Freed-Lott differential $K$-theory and the construction of the Freed-Lott
differential analytic index. The main results of the paper are proved in Section 5,
including the explicit isomorphisms between Simons-Sullivan differential $K$-theory
and Freed-Lott differential $K$-theory, the formula for the differential analytic
index in Simons-Sullivan differential $K$-theory and the dGRR.
\section*{Acknowledgement}
The author would like to thank Steven Rosenberg for many stimulating discussions of this
problem and the referee for many helpful suggestions.
\section{Simons-Sullivan differential $K$-theory}
In this section we review Simons-Sullivan differential $K$-theory \cite{SS10}. For
our purpose, we use the Hermitian version of structured bundles instead of the complex
version. Consider a triple $(V, h, \nabla)$, where $V\to X$ is a Hermitian vector
bundle over a compact manifold $X$ with a Hermitian metric $h$ and a unitary connection
$\nabla$. Recall that the Chern character form $\ch(\nabla)\in\Omega^{\even}(X; \R)$
and the Chern-Simons transgression form $\cs(\nabla^t)\in\Omega^{\odd}(X; \R)$ of two
connections $\nabla^0$, $\nabla^1$ on $V\to X$ joined by a smooth curve $\nabla^t$ of
connections are related by the equality
\begin{equation}\label{eq 5}
d\cs(\nabla^t)=\ch(\nabla^1)-\ch(\nabla^0).
\end{equation}
Define
$$\CS(\nabla^0, \nabla^1):=\cs(\nabla^t)\mod\im(d:\Omega^{\even}(X)\to\Omega^{\odd}(X)),$$
where $\nabla^t$ is a smooth curve joining the connection $\nabla^1$ and $\nabla^0$.
Since $\cs(\nabla^t)$ only depends on the curve joining the connections up to an
exact form \cite[Proposition 1.6]{SS10}, $\CS(\nabla^0, \nabla^1)$ is well defined
\footnote{It follows from (\ref{eq 5}) that $d\CS(\nabla^0, \nabla^1)=\ch(\nabla^1)-
\ch(\nabla^0)$. There are other sign convention, for example see \cite{FL10}. We will
use the convention $d\CS(\nabla^0, \nabla^1)=\ch(\nabla^1)-\ch(\nabla^0)$ in this paper.}.\\\\
For two connections $\nabla^0$, $\nabla^1$ on $V\to X$, we set $\nabla^0\sim\nabla^1$ if
and only if $\CS(\nabla^0, \nabla^1)=0$. $\sim$ is an equivalence relation.\\\\
The triple $\V=(V, h, [\nabla])$ is called a (Hermitian) structured bundle. Two
structured bundles $\V=(V, h^V, [\nabla^V])$ and $\W=(W, h^W, [\nabla^W])$ are isomorphic
if there exists a vector bundle isomorphism $\sigma:V\to W$ such that $\sigma^*h^W=h^V$
and $\sigma^*([\nabla^W])=[\nabla^V]$. Denote by $\st(X)$ the set of all isomorphism
classes of structured bundles. Direct sum and tensor product of structured bundles are
well-defined \cite{SS10}, so $\st(X)$ is an abelian semi-ring.\\\\
The Simons-Sullivan differential $K$-group is defined to be
$$\wh{K}_{\SSSS}(X)=K(\st(X)).$$
Thus, Simons-Sullivan differential $K$-theory is a $K$-theory of vector bundles with
connections.\\\\
To be precise, $[\V_1]-[\W_1]=[\V_2]-[\W_2]$ in $\wh{K}_{\SSSS}(X)$ if and only if
there exists a structured bundle $(G, h^G, [\nabla^G])\in\st(X)$ such that $V_1\oplus
W_2\oplus G\cong W_1\oplus V_2\oplus G$ and $\CS(\nabla^{V_1}\oplus\nabla^{W_2}\oplus
\nabla^G, \nabla^{V_2}\oplus\nabla^{W_1}\oplus\nabla^G)=0$.\\\\
Define
\begin{displaymath}
\begin{split}
\st_{\ST}(X)&=\set{\V\in\st(X)|V\textrm{ is stably trivial}}\\
\st_{\SF}(X)&=\set{\V\in\st(X)|\V\oplus\F\cong\HH}
\end{split}
\end{displaymath}
where $\F\to X$ and $\HH\to X$ are flat structured bundles. Elements in $\st_{\SF}(X)$
are said to be stably flat. Let $\U:=\dirlim\U(n)$. Denote by $\theta\in\Omega^1(\U,
\uu)$ the canonical left invariant $\uu$-valued form on $\U$. Define
\begin{displaymath}
\begin{split}
b_j&=\frac{1}{(j-1)!}\bigg(\frac{1}{2\pi i}\bigg)^j\int^1_0(t^2-t)^{j-1}dt, j\in\N\\
\Theta&=\sum_{j=1}b_j\tr(\overbrace{\theta\wedge\cdots\wedge\theta}^{2j-1})\in
\Omega^{\odd}(\U)
\end{split}
\end{displaymath}
Then define
\begin{displaymath}
\begin{split}
\Omega_{\U}(X)&=\set{g^*(\Theta)+d\beta|g:X\to\U, \beta\in\Omega^{\even}(X)}\\
\Omega^\bullet_{\BU}(X)&=\set{\omega\in\Omega^\bullet_{d=0}(X)|[\omega]\in\im(\ch:
K^{-(\bullet\mod 2)}(X)\to H^\bullet(X; \Q))}.
\end{split}
\end{displaymath}
where $\bullet\in\set{\even, \odd}$. The so-called Venice lemma in \cite{SS10} shows
that the map \dis{\wh{\CS}:\frac{\st_{\ST}(X)}{\st_{\SF}(X)}\to\frac{\Omega^{\odd}
(X)}{\Omega_{\U}(X)}} defined by\footnote{This definition differs from the one
in \cite[Proposition 2.4]{SS10} by a sign.}
$$\wh{\CS}(\V):=\CS(\nabla^V\oplus\nabla^F, \nabla^H)\mod\frac{\Omega_{\U}
(X)}{\Omega^{\odd}_{\exact}(X)}$$
is an isomorphism, where $F\to X$ and $H\to X$ are trivial bundles over $X$ such that
$H\cong V\oplus F$ and $\nabla^F$, $\nabla^H$ are flat connections on $F$, $H$,
respectively. Also, the homomorphism
$$\Gamma:\frac{\st_{\ST}(X)}{\st_{\SF}(X)}\to\wh{K}_{\SSSS}(X)$$
defined by $\Gamma(\V)=[\V]-[\dim(\V)]$ is injective, for $\dim(\V)$ the trivial
structured bundle of rank $V$ with the trivial metric and connection. Thus the
homomorphism
$$i:\frac{\Omega^{\odd}(X)}{\Omega_{\U}(X)}\to\wh{K}_{\SSSS}(X)$$
defined by $i(\phi)=\Gamma\circ\wh{\CS}^{-1}(\phi)$ is injective. If we pick
$\V\in\wh{\CS}^{-1}(\phi)$, then $\V$ is a stably trivial structured bundle and
$$d\phi=d\CS(\nabla^V\oplus\nabla^F, \nabla^H)=\ch(\nabla^V)-\rk(V)\mod
\frac{\Omega_{\U}(X)}{\Omega^{\odd}_{\exact}(X)}$$
is independent of the choice of $\V$.\\\\
In the following hexagon the diagonal and the off-diagonal sequences are exact, and
every square and triangle commutes:
\begin{equation}\label{eq a}
\xymatrix{\scriptstyle 0 \ar[dr] & \scriptstyle & \scriptstyle & \scriptstyle &
\scriptstyle 0 \\ & \scriptstyle K^{-1}_{\SSSS}(X; \R/\Z) \ar[rr]^{B} \ar[dr]^j &
\scriptstyle & \scriptstyle K(X) \ar[ur] \ar[dr]^{\ch} & \scriptstyle \\ \scriptstyle
H^{\odd}(X; \R) \ar[ur]^{i\circ\der} \ar[dr]_{\der} & \scriptstyle & \scriptstyle
\wh{K}_{\SSSS}(X) \ar[ur]^{\delta} \ar[dr]^{\ch_{\wh{K}_{\SSSS}}} & \scriptstyle &
\scriptstyle H^{\even}(X; \R) \\ \scriptstyle & \scriptstyle
\frac{\Omega^{\odd}(X)}{\Omega_{\U}(X)} \ar[rr]_{d} \ar[ur]^i & \scriptstyle
& \scriptstyle \Omega^{\even}_{\BU}(X) \ar[ur]_{\der} \ar[dr] & \scriptstyle \\
\scriptstyle 0 \ar[ur] & \scriptstyle & \scriptstyle & \scriptstyle & \scriptstyle 0}
\end{equation}
In \cite{SS10} the homomorphism $\ch_{\wh{K}_{\SSSS}}:\wh{K}_{\SSSS}(X)\to
\Omega^{\even}_{\BU}(X)$ is just denoted by $\ch$, which is a well defined lift of the
Chern character form of a connection on a vector bundle to elements in $\wh{K}_{\SSSS}
(X)$. We use the notation $\ch_{\wh{K}_{\SSSS}}$ in order to keep track of the Chern
character in different usage.
\begin{remark}\label{remarkiai 1}
{\emph{We show that $\Omega_{\U}(X)=\Omega^{\odd}_{\BU}(X)$, and we will use
this identification throughout this paper. This is implicitly stated in
\cite[Diagram 1]{SS10a}. We include the easy proof here for completeness. Let $d$ be
the trivial connection on the trivial bundle $X\times\C^N\to X$ for some $N\in\N$.
By the proof of \cite[Lemma 2.3]{SS10}, the connection $d+g^*(\theta)$ on $X\times
\C^N\to X$, where $g:X\to\U$ is an arbitrary but fixed smooth map, has trivial holonomy.
Following the proof of \cite[Lemma 2.3]{SS10}, we have $g^*(\Theta)=\CS(d, d+g^*
(\theta))=\CS(d, d+g^{-1}dg)=:\ch^{\odd}([g])$, so $\Omega_{\U}(X)\Omega^{\odd}_{\BU}
(X)$.}}
\end{remark}
\section{Freed-Lott differential $K$-theory}
\noindent In this section we review Freed-Lott differential $K$-theory \cite{FL10}.
If
\begin{equation}\label{eq 1}
\xymatrix{0 \ar[r] & E_1 \ar[r]^i & E_2 \ar@/^/[r]^j & E_3 \ar[r] \ar@/^/[l]^s & 0}
\end{equation}
is a split short exact sequence of complex vector bundles with connections $\nabla_i$
on $E_i\to X$, for $i=1, 2, 3$, we define the relative Chern-Simons transgression form
\dis{\CS(\nabla_1, \nabla_2, \nabla_3)\in\frac{\Omega^{\odd}(X)}{\im(d)}} by
$$\CS(\nabla_1, \nabla_2, \nabla_3):=\CS((i\oplus s)^*\nabla_2, \nabla_1\oplus\nabla_3),$$
noting that $i\oplus s:E_1\oplus E_3\to E_2$ is a vector bundle isomorphism.\\\\
The Freed-Lott differential $K$-group $\wh{K}_{\FL}(X)$ is defined to be the abelian
group with the following generators and relation: a generator of $\wh{K}_{\FL}(X)$ is
a quadruple $\E=(E, h, \nabla, \phi)$, where $(E, h, \nabla)$ is as before and
\dis{\phi\in\frac{\Omega^{\odd}(X)}{\im(d)}}. The only relation is $\E_2=\E_1+\E_3$
if and only if there exists a short exact sequence of Hermitian vector bundles
(\ref{eq 1}) and
$$\phi_2=\phi_1+\phi_3-\CS(\nabla_1, \nabla_2, \nabla_3).$$
For $\E_1, \E_2\in\wh{K}_{\FL}(X)$, the addition
$$\E_1+\E_2:=(E_1\oplus E_2, h^{E_1}\oplus h^{E_2}, \nabla^{E_1}\oplus\nabla^{E_2},
\phi_1+\phi_2)$$
is well defined. Note that $\E_1=\E_2$ if and only if there exists $(F, h^F, \nabla^F,
\phi^F)\in\wh{K}_{\FL}(X)$ such that
\begin{enumerate}
  \item $E_1\oplus F\cong E_2\oplus F$, and
  \item $\phi_1-\phi_2=\CS(\nabla^{E_2}\oplus\nabla^F, \nabla^{E_1}\oplus\nabla^F)$,
\end{enumerate}
The Freed-Lott differential Chern character $\wh{\ch}_{\FL}:\wh{K}_{\FL}(X)\to
\wh{H}^{\even}(X; \R/\Q)$ is defined by
$$\wh{\ch}_{\FL}(\E)=\wh{\ch}(E, \nabla)+i_2(\phi),$$
where $\wh{H}^{\even}(X; \R/\Q)$ is the $\R/\Q$ Cheeger-Simons differential characters
\cite{CS85}, $\E=(E, h, \nabla, \phi)\in\wh{K}_{\FL}(X)$, $\wh{\ch}(E, \nabla)$ is the
differential Chern character defined in \cite[\S 4]{CS85}, and \dis{i_2:
\frac{\Omega^{\odd}(X)}{\Omega^{\odd}_\Q(X)}\to\wh{H}^{\even}(X; \R/\Q)} is an injective
homomorphism defined by \dis{i_2(\omega)(z):=\int_z\omega\mod\Q} for $z\in Z_{\even}(X)$
\cite[Theorem 1.1]{CS85}.
\subsection{The Freed-Lott differential analytic index}
In this subsection we review the construction of the Freed-Lott differential analytic
index. Consider the following diagram:
$$\xymatrix{\scriptstyle (E, h, \nabla) \ar[ddr] & \scriptstyle(S^VX, \wh{\nabla}^{T^VX})
\ar@/^/[dd] & \scriptstyle(L^VX, \nabla^{L^VX}) \ar[ddl] & \scriptstyle\pi_*E \ar@/^/[dd]
\\ \scriptstyle(TX, g^{TX}, \nabla^{TX}) \ar[dr] & \scriptstyle(T^VX, g^{T^VX},
\nabla^{T^VX})\ar[d] & \scriptstyle (T^HX, \pi^*g^{TB}) \ar[dl] & \scriptstyle(\ker(D^V),
h^{\ker(D^V)}, \nabla^{\ker(D^V)} \ar[d] & \\ & \scriptstyle X \ar[rr]_\pi & & \scriptstyle
(B, g^{TB})}$$
In this diagram, $\pi:X\to B$ is a proper submersion with closed fibers of even relative
dimension and $T^VX\to X$ is the vertical tangent bundle, which is assumed to have a metric
$g^{T^VX}$. $T^HX\to X$ is a horizontal distribution, $g^{TB}$ is a Riemannian metric on
$B$, the metric on $TX\to X$ is defined by $g^{TX}:=g^{T^VX}\oplus\pi^*g^{TB}$, $\nabla^{TX}$
is the corresponding Levi-Civita connection, and $\nabla^{T^VX}:=P\circ\nabla^{TX}\circ P$
is a connection on $T^VX\to X$, where $P:TX\to T^VX$ is the orthogonal projection.
$T^VX\to X$ is assumed to have a $\spin^c$ structure. Denote by $S^VX\to X$ the
$\spin^c$-bundle associated to the characteristic Hermitian line bundle $L^V\to X$ with
a unitary connection. The connections on $T^VX\to X$ and $L^VX\to X$ induce a connection
$\wh{\nabla}^{T^VX}$ on $S^VX\to X$. Define an even form $\todd(\wh{\nabla}^{T^VX})
\in\Omega^{\even}(X)$ by
$$\todd(\wh{\nabla}^{T^VX})=\wh{A}(\nabla^{T^VX})\wedge e^{\frac{1}{2}c_1
(\nabla^{L^VX})}.$$
The modified pushforward of forms $\pi_*:\Omega^{\odd}(X)\to\Omega^{\odd}(B)$ is defined
by
$$\pi_*(\phi)=\int_{X/B}\todd(\wh{\nabla}^{T^VX})\wedge\phi.$$
The Freed-Lott differential analytic index $\ind^{\an}:\wh{K}_{\FL}(X)\to\wh{K}_{\FL}
(B)$ \cite[Definition 3.11]{FL10} is defined by
$$\ind^{\an}(\E)=(\ker(D^E), h^{\ker(D^E)}, \nabla^{\ker(D^E)}, \pi_*(\phi)+\wt{\eta}(\E)),$$
where $\E=(E, h, \nabla, \phi)\in\wh{K}_{\FL}(X)$, $\wt{\eta}(\E)$ is the Bismut-Cheeger
eta form \cite{BC86} characterized, up to exact form, by
$$d\wt{\eta}(\E)=\int_{X/B}\todd(\wh{\nabla}^{T^VX})\wedge\ch(\nabla)-\ch(\nabla^{\ker(D^E)}),$$
$D^E$ is the family of Dirac operators on $S^VX \otimes E$, and $\ker(D^E)$ is assumed to
form a superbundle over $B$.
\section{Main results}
\subsection{Explicit isomorphisms between $\wh{K}_{\FL}$ and $\wh{K}_{\SSSS}$}
In this subsection we construct explicit isomorphisms between the Simons-Sullivan
differential $K$-group and the Freed-Lott differential $K$-group.
\begin{thm}\label{thm 1}
Let $X$ be a compact manifold. The maps
$$f:\wh{K}_{\SSSS}(X)\to\wh{K}_{\FL}(X),\qquad g:\wh{K}_{\FL}(X)\to\wh{K}_{\SSSS}(X)$$
defined by
\begin{displaymath}
\begin{split}
f([E, h^E, [\nabla^E]]-[F, h^F, [\nabla^F]])&=(E, h^E, \nabla^E, 0)-(F, h^F, \nabla^F, 0),\\
g(E, h^E, \nabla^E, \phi)&=[E, h^E, [\nabla^E]]+[V, h^V, [\nabla^V]]-[\dim(V), h, [d]],
\end{split}
\end{displaymath}
where $\V=(V, h^V, [\nabla^V])\in\wh{\CS}^{-1}(\phi)$, are well defined ring
isomorphisms, with $f^{-1}=g$. Moreover, $f$ is natural and unique
\cite[Theorem 3.10]{BS10a}, and is compatible with the structure maps $i$, $j$, $\delta$
and $\ch_{\wh{K}_{\SSSS}}$ in (\ref{eq a}).
\end{thm}
\begin{proof}
First we show that the maps $f$ and $g$ are well defined. For the map $f$, if
$[E_1, h^{E_1}, [\nabla^{E_1}]]-[F_1, h^{F_1}, [\nabla^{F_1}]]=[E_2, h^{E_2},
[\nabla^{E_2}]]-[F_2, h^{F_2}, [\nabla^{F_2}]]$ in $\wh{K}_{\SSSS}(X)$, then
$$(E_1, h^{E_1}, \nabla^{E_1}, 0)-(F_1, h^{F_1}, \nabla^{F_1}, 0)=(E_2, h^{E_2},
\nabla^{E_2}, 0)-(F_2, h^{F_2}, \nabla^{F_2}, 0),$$
since there exists $(G, h^G, [\nabla^G])\in\st(X)$ such that
$E_1\oplus F_2\oplus G\cong F_1\oplus E_2\oplus G$ and
$$0=\CS(\nabla^{E_1}\oplus\nabla^{F_2}\oplus\nabla^G, \nabla^{F_1}\oplus\nabla^{E_2}
\oplus\nabla^G)=\CS(\nabla^{E_1}\oplus\nabla^{F_2}, \nabla^{F_1}\oplus\nabla^{E_2}).$$
It follows that the map $f$ is well defined.\\\\
For the map $g$, if $(E, h^E, \nabla^E, \phi)=(F, h^F, \nabla^F, \omega)$ in
$\wh{K}_{\FL}(X)$, then there exists $(G, h^G, \nabla^G, \phi^G)\in\wh{K}_{\FL}(X)$
such that $E\oplus G\cong F\oplus G$ and $\phi-\omega=\CS(\nabla^F\oplus\nabla^G,
\nabla^E\oplus\nabla^G)$. We want
\begin{displaymath}
\begin{split}
&~~~~[E, h^E, [\nabla^E]]+[V, h^V, [\nabla^V]]-[\dim(V), h, [d]]\\
&=[F, h^F, [\nabla^F]]+[W, h^W, [\nabla^W]]-[\dim(W), h, [d]],
\end{split}
\end{displaymath}
where $\wh{\CS}(\V)=\phi$ and $\wh{\CS}(\W)=\omega$. We need to show that there
exists $(G', h^{G'}, [\nabla^{G'}])\in\st(X)$ such that
\begin{equation}\label{eq 1'}
\begin{split}
&~~~~(E, h^E, [\nabla^E])+(V, h^V, [\nabla^V])+(\dim(W), h, [d])+(G', h^{G'}, [\nabla^{G'}])\\
&=(F, h^F, [\nabla^F])+(W, h^W,[\nabla^W])+(\dim(V), h, [d])+(G', h^{G'}, [\nabla^{G'}]),
\end{split}
\end{equation}
and $\CS(\nabla^E\oplus\nabla^V\oplus d^W\oplus\nabla^{G'}, \nabla^F\oplus\nabla^W
\oplus d^V\oplus\nabla^{G'})=0$. (\ref{eq 1'}) is equivalent to
\begin{displaymath}
\begin{split}
&~~~~(E\oplus V\oplus\dim(W)\oplus G', h^E\oplus h^V\oplus h\oplus h^{G'}, [\nabla^E\oplus
\nabla^V\oplus d\oplus\nabla^{G'}])\\
&=(F\oplus W\oplus\dim(V)\oplus G', h^F\oplus h^W\oplus h\oplus h^{G'}, [\nabla^F\oplus
\nabla^W\oplus d\oplus\nabla^{G'}]).
\end{split}
\end{displaymath}
Since $V$ and $W$ are stably trivial, there exist trivial bundles $V'$ and $W'$ with
connections $\nabla^{V'}$ and $\nabla^{W'}$ such that
$$H^V:=\dim(V)\oplus V'=V\oplus V',\qquad H^W:=\dim(W)\oplus W'=W\oplus W',$$
and
$$\phi=\CS(\nabla^V\oplus\nabla^{V'}, \nabla^{H^V}),\qquad\omega=\CS(\nabla^W\oplus\nabla^{W'},
\nabla^{H^W}).$$
By taking $G'=G\oplus V'\oplus W'$, we have
\begin{equation}\label{eq 2}
\begin{split}
&~~~~(E\oplus V\oplus\dim(W))\oplus(G\oplus V'\oplus W')\\
&\cong(E\oplus G)\oplus(V\oplus V')\oplus(\dim(W)\oplus W')\\
&\cong(F\oplus G)\oplus(\dim(V)\oplus V')\oplus(W\oplus W')\\
&\cong(F\oplus W\oplus\dim(V))\oplus(G\oplus V'\oplus W')
\end{split}
\end{equation}
and for $d^V$, $d^W$ the trivial connections on $\dim(V)$, $\dim(V)$, respectively,
\begin{equation}\label{eq 3}
\begin{split}
&~~~~\CS(\nabla^E\oplus\nabla^V\oplus d^W\oplus\nabla^G\oplus\nabla^{V'}\oplus\nabla^{W'},
\nabla^F\oplus\nabla^W\oplus d^V\oplus\nabla^G\oplus\nabla^{V'}\oplus\nabla^{W'})\\
&=\CS(\nabla^E\oplus\nabla^G, \nabla^F\oplus\nabla^G)+\CS(\nabla^V\oplus\nabla^{V'},
d^V\oplus\nabla^{V'})+\CS(d^W\oplus\nabla^{W'}, \nabla^W\oplus\nabla^{W'})\\
&=-\phi+\omega+\CS(\nabla^V\oplus\nabla^{V'}, \nabla^{H^V})+\CS(\nabla^{H^W}, \nabla^W
\oplus\nabla^{W'})\\
&=-\phi+\omega+\phi-\omega\\
&=0
\end{split}
\end{equation}
(\ref{eq 2}) and (\ref{eq 3}) imply (\ref{eq 1'}), so the map
$g:\wh{K}_{\FL}(X)\to\wh{K}_{\SSSS}(X)$ is well defined.\\\\
We now show that $f$ and $g$ are inverses. Note that
\begin{displaymath}
\begin{split}
(g\circ f)([E, h^E, [\nabla^E]]-[F, h^F, [\nabla^F]])&=g((E, h^E, \nabla^E, 0)-(F, h^F,
\nabla^F, 0))\\
&=[E, h^E, [\nabla^E]]-[F, h^F, [\nabla^F]]
\end{split}
\end{displaymath}
as \dis{\wh{\CS}^{-1}(0)=0\in\frac{\st_{\ST}(X)}{\st_{\SF}(X)}}. For the other direction,
we consider
$$(f\circ g)(E, h^E, \nabla^E, \phi)=(E, h^E, \nabla^E, 0)+(V, h^V, \nabla^V, 0)-(\dim(V),
h, d, 0),$$
where $\wh{\CS}(\V)=\phi$ for $\V:=(V, h^V, [\nabla^V])\in\st_{\ST}(X)$. It suffices
to show
$$(E, h^E, \nabla^E, \phi)+(\dim(V), h, d^V, 0)=(E, h^E, \nabla^E, 0)+(V, h^V, \nabla^V,
0),$$
which is equivalent to
\begin{equation}\label{eq 3'}
(E\oplus\dim(V), h^E\oplus h, \nabla^E\oplus d^V, \phi)=(E\oplus V, h^E\oplus h^V, \nabla^E
\oplus\nabla^V, 0).
\end{equation}
To see this, since $\V=(V, h^V, [\nabla^V])$ is stably trivial, there exist trivial
structured bundles $\F=(F, h, [d^F])$ and $\HH=(H, h, [d^H])$ such that $V\oplus F\cong H$
and $\phi=\CS(d^H, \nabla^V\oplus d^F)$. Thus
$E\oplus\dim(V)\oplus\dim(F)\cong E\oplus\dim(H)\cong E\oplus V\oplus F$, and
\begin{displaymath}
\begin{split}
\CS(\nabla^E\oplus\nabla^V, \nabla^E\oplus d^V)&=\CS(\nabla^E\oplus\nabla^V\oplus d^F,
\nabla^E\oplus d^V\oplus d^F)\\
&=\CS(\nabla^E\oplus\nabla^V\oplus d^F, \nabla^E\oplus d^H)=\phi.
\end{split}
\end{displaymath}
This proves (\ref{eq 3'}).\\\\
$f$ is obviously a natural ring homomorphism. Since $g=f^{-1}$, $g$ is also a ring
homomorphism.
\end{proof}
The following corollary follows from the compatibility of $f$ and $\ch_{\wh{K}_{\SSSS}}$.
\begin{coro}\label{coro 1}
Let $X$ be a compact manifold. The following diagram is commutative.
\cdd{0 @>>> K^{-1}_{\SSSS}(X; \R/\Z) @>j>> \wh{K}_{\SSSS}(X) @>\ch>> \Omega_{\BU}(X) @>>>
0 \\ & & @V\bar{f}VV @VfVV @V=VV \\ 0 @>>> K^{-1}_{\LL}(X; \R/\Z) @>>j'> \wh{K}_{\FL}(X)
@>>\omega> \Omega_{\BU}(X) @>>> 0}
where $\bar{f}$ is the restriction of $f$ to $K^{-1}_{\SSSS}(X; \R/\Z)$. Here
$\omega:\wh{K}_{\FL}(X)\to\Omega_{\BU}(X)$ is defined by $\omega(E, h^E, \nabla^E, \phi)
=\ch(\nabla^E)+d\phi$.
\end{coro}
Note that the horizontal sequences are exact by \cite{FL10}, \cite{SS10}.
\subsection{The differential analytic index in $\wh{K}_{\SSSS}$}
In this subsection we give the formula for the differential analytic index in
Simons-Sullivan differential $K$-theory.\\\\
Let $\pi:X\to B$ be a proper submersion of even relative dimension and its fibers
are assumed to be $\spin^c$. The differential analytic index in Simons-Sullivan
differential $K$-theory is given by forcing the following diagram to be commutative:
\cdd{\wh{K}_{\SSSS}(X) @>f>> \wh{K}_{\FL}(X) \\ @V\ind^{\an}_{\SSSS}VV
@VV\ind^{\an}_{\FL}V \\ \wh{K}_{\SSSS}(B) @<<g< \wh{K}_{\FL}(B)}
Let $\E:=[E, h^E, [\nabla]]\in\wh{K}_{\SSSS}(X)$. Since
\begin{displaymath}
\begin{split}
(g\circ\ind^{\an}_{\FL}\circ f)(\E)&=[\ker(D^E), h^{\ker(D^E)}, [\nabla^{\ker(D^E)}]]\\
&\qquad+[V, h^V, [\nabla^V]]-[\dim(V), h, [d]],
\end{split}
\end{displaymath}
where \dis{\V:=(V, h^V, [\nabla^V])\in\frac{\st_{\ST}(B)}{\st_{\SF}(B)}} is uniquely
determined by the condition \dis{\wh{\CS}(\V)=\wt{\eta}(\E)\mod\frac{\Omega_{\U}
(B)}{\Omega^{\odd}_{\exact}(B)}}, it follows that the differential analytic index in
the Simons-Sullivan differential $K$-theory $\ind^{\an}_{\SSSS}:\wh{K}_{\SSSS}(X)\to
\wh{K}_{\SSSS}(B)$ is given by
\begin{equation}\label{eq 21}
\ind^{\an}_{\SSSS}(\E)=[\ker(D^E), h^{\ker(D^E)}, [\nabla^{\ker(D^E)}]]
+[V, h^V, [\nabla^V]]-[\dim(V), h, [d]],
\end{equation}
where $\ker(D^E)$ is assumed to form a superbundle over $B$. Although
$\V:=\wh{\CS}^{-1}(\wt{\eta}(\E))$ is uniquely determined up to a stably flat structured
bundle, its class $[\V]\in\wh{K}_{\SSSS}(B)$ is unique since the differential $K$-theory
class of a stably flat structured bundle is zero. Moreover, since $\ind^{\an}_{\FL}$ is
well defined (see \cite[Proposition 1]{H11a} for a proof which does not use the
differential family index theorem), it follows that $\ind^{\an}_{\SSSS}$ is well defined
too.\\\\
If one defines the Simons-Sullivan differential analytic index $\ind^{\an}_{\SSSS}$
without considering the other differential analytic indices, a natural candidate
would be, say, in the special case when $\ker(D^E)\to B$ is a superbundle,
$$\ind^{\an}_{\SSSS}(\E)=[\ker(D^E), h^{\ker(D^E)}, [\nabla^{\ker(D^E)}]].$$
This definition coincides with (\ref{eq 21}) if and only if $\V\in\st_{\SF}(B)$, which
is equivalent to saying that $\wt{\eta}(\E)\in\Omega_{\U}(B)=\Omega^{\odd}_{\BU}(B)$.
However, this is not true since
$$d\wt{\eta}(\E)=\int_{X/B}\todd(\wh{\nabla}^{T^VX})\wedge\ch(\nabla)-\ch(\nabla^{\ker
(D^E)}),$$
which shows that $\wt{\eta}(\E)$ is not closed in general.
\begin{lemma}\label{lemma 1}
Let $\E=[E, h, [\nabla]]\in\wh{K}_{\SSSS}(X)$. Then
$$\ch_{\wh{K}_{\SSSS}}(\ind^{\an}_{\SSSS}(\E))=\ch(\nabla^{\ker(D^E)})+d\wt{\eta}(\E).$$
\end{lemma}
It follows from Lemma \ref{lemma 1} and the local family index theorem that
\begin{displaymath}
\begin{split}
\ch_{\wh{K}_{\SSSS}}(\ind^{\an}_{\SSSS}(\E))&=\ch(\nabla^{\ker(D^E)})+d\wt{\eta}(\E)\\
&=\int_{X/B}\todd(\wh{\nabla}^{T^VX})\wedge\ch(\nabla^E)\\
&=\pi_*(\ch_{\wh{K}_{\SSSS}}(\E)).
\end{split}
\end{displaymath}
We define the Simons-Sullivan differential Chern character $\wh{\ch}_{\SSSS}:\wh{K}_{\SSSS}
(X)\to\wh{H}^{\even}(X; \R/\Q)$ by
$$\wh{\ch}_{\SSSS}(\E):=\wh{\ch}(E, \nabla),$$
where $\E=[E, h, [\nabla]]$.\\\\
It is instructive to note that the following diagram commute,
$$\xymatrix{ & \wh{K}_{\SSSS}(X) \ar[dd]^f \\ \displaystyle\frac{\Omega^{\odd}
(X)}{\Omega^{\odd}_{\BU}(X)} \ar[ur]^i \ar[dr]_j & \\ & \wh{K}_{\FL}(X)}$$
where $f:\wh{K}_{\SSSS}(X)\to\wh{K}_{\FL}(X)$ is the isomorphism given by Theorem \ref{thm 1}.
\subsection{The differential Grothendieck-Riemann-Roch theorem}
In this subsection we prove the dGRR in Simons-Sullivan differential $K$-theory. To be
precise, we first prove the special case that the family of kernels $\ker(D^E)$ forms
a superbundle by a theorem of Bismut reviewed below. The general case follows from the
standard perturbation argument as in \cite[\S 7]{FL10}.\\\\
We now recall Bismut's theorem. For the geometric construction of the analytic index
given in \S 4.2, with the fibers assumed to be $\spin$, and $\ker(D^E)\to B$ assumed to
form a superbundle, we have
\begin{equation}\label{eqiai 11}
\wh{\ch}(\ker(D^E), \nabla^{\ker(D^E)})+i_2(\wt{\eta})=\wh{\int_{X/B}}
\wh{\wh{A}}(T^VX, \nabla^{T^VX})\ast\wh{\ch}(E, \nabla^E)
\end{equation}
\cite[Theorem 1.15]{B05}, where \dis{\wh{\int_{X/B}}} is the pushforward of
differential characters for proper submersion \cite[\S 8.3]{FL10}, $\ast$ is the
multiplication of differential characters \cite[\S 1]{CS85}, and $\wh{\wh{A}}(T^VZ,
\nabla^{T^VX})\in\wh{H}^{\even}(X; \R/\Q)$ is the differential character associated to
the $\wh{A}$-class (see \cite[\S 2]{CS85}). If the fibers are assumed to be $\spin^c$,
(\ref{eqiai 11}) has the obvious modification, and in our notation becomes
\begin{equation}\label{eqiai 12}
\wh{\ch}(\ker(D^E), \nabla^{\ker(D^E)})+i_2(\wt{\eta})=\wh{\int_{X/B}}\wh{\todd}(T^VX,
\wh{\nabla}^{T^VX})\ast\wh{\ch}(E, \nabla^E),
\end{equation}
for $\wh{\todd}(T^VX, \wh{\nabla}^{T^VX})\in\wh{H}^{\even}(X; \R/\Q)$ the differential
character associated to the Todd class (see \cite[\S 2]{CS85}). We will write
$\wh{\todd}(T^VX, \wh{\nabla}^{T^VX})$ as $\wh{\todd}(\wh{\nabla}^{T^VX})$ in the
sequel. Note that (\ref{eqiai 11}) and (\ref{eqiai 12}) extend to the general case
where $\ker(D^E)\to B$ does not form a bundle \cite[p. 23]{B05}.
\begin{thm}[\bf Differential Grothendieck-Riemann-Roch theorem]\label{iaithm 3}
Let $\pi:X\to B$ be a proper submersion with closed $\spin^c$-fibers of even dimension.
The following diagram is commutative:
\cdd{\wh{K}_{\SSSS}(X) @>\wh{\ch}_{\SSSS}>> \wh{H}^{\even}(X; \R/\Q) \\
@V\ind^{\an}_{\SSSS}VV @VV\wh{\int_{X/B}}\wh{\todd}(\wh{\nabla}^{T^VX})\ast(\cdot) V \\
\wh{K}_{\SSSS}(B) @>>\wh{\ch}_{\SSSS}> \wh{H}^{\even}(B; \R/\Q)}
i.e., if $\E=[E, h, [\nabla^E]]\in\wh{K}_{\SSSS}(X)$, then
$$\wh{\ch}_{\SSSS}(\ind^{\an}_{\SSSS}(\E))=\wh{\int_{X/B}}\wh{\todd}(\nabla^{T^VX})
\ast\wh{\ch}_{\SSSS}(\E).$$
\end{thm}
\begin{proof}
\begin{displaymath}
\begin{split}
&\qquad\wh{\ch}_{\SSSS}(\ind^{\an}_{\SSSS}(\E))\\
&=\wh{\ch}_{\SSSS}([\ker(D^E), h^{\ker(D^E)}, [\nabla^{\ker(D^E)}]]+[V, h^V, [\nabla^V]]
-[\dim(V), h, [d]])\\
&=\wh{\ch}(\ker(D^E), \nabla^{\ker(D^E)})+i_2(\wt{\eta}(\E))\\
&=\wh{\int_{X/B}}\wh{\todd}(\wh{\nabla}^{T^VX})\ast\wh{\ch}(E, \nabla^E)\\
&=\wh{\int_{X/B}}\wh{\todd}(\wh{\nabla}^{T^VX})\ast\wh{\ch}_{\SSSS}(\E)\\
\end{split}
\end{displaymath}
where the second equality follows from Proposition \ref{iaiprop 4} and the third
equality follows from (\ref{eqiai 12}).
\end{proof}
\bibliographystyle{amsplain}
\bibliography{MBib}

\providecommand{\bysame}{\leavevmode\hbox to3em{\hrulefill}\thinspace}
\providecommand{\MR}{\relax\ifhmode\unskip\space\fi MR }
\providecommand{\MRhref}[2]{%
  \href{http://www.ams.org/mathscinet-getitem?mr=#1}{#2}
}
\providecommand{\href}[2]{#2}
\begin{thebibliography}{10}

\bibitem{B05}
J.M. Bismut, \emph{Eta invariants, differential characters and flat vector
  bundles}, Chinese Ann. Math. Ser. B \textbf{26} (2005), 15--44.

\bibitem{BC86}
J.M. Bismut and J.~Cheeger, \emph{$\eta$-invariants and their adiabatic
  limits}, J. Amer. Math. Soc. \textbf{2} (1989), 33--70.

\bibitem{BS09}
U.~Bunke and T.~Schick, \emph{Smooth {K}-theory}, Ast$\acute{\textrm{e}}$risque
  \textbf{328} (2009), 45--135.

\bibitem{BS10a}
\bysame, \emph{Uniqueness of smooth extensions of generalized cohomology
  theories}, J. Topol. \textbf{3} (2010), 110--156.

\bibitem{BS10}
\bysame, \emph{Differential {K}-theory. {A} survey}, Global Differential
  Geometry (Berlin Heidelberg) (C.~B$\ddot{\textrm{a}}$r, J.~Lohkamp, and
  M.~Schwarz, eds.), Springer Proceedings in Mathematics, vol.~17,
  Springer-Verlag, 2012, pp.~303--358.

\bibitem{CS85}
J.~Cheeger and J.~Simons, \emph{Differential characters and geometric
  invariants}, in Geometry and Topology (College Park, Md., 1983/84), Lecture
  Notes in Math. \textbf{1167} (1985), 50--80.

\bibitem{DF00}
D.~Freed, \emph{Dirac charge quantization and generalized differential
  cohomology}, Surv. Diff. Geom., VII, Int. Press, Somerville, MA \textbf{1}
  (2000), 129--194.

\bibitem{FL10}
D.~Freed and J.~Lott, \emph{An index theorem in differential {K}-theory}, Geom.
  Topol. \textbf{14} (2010), 903--966.

\bibitem{H11a}
M.H. Ho, \emph{A short proof of the differential
  {G}rothendieck-{R}iemann-{R}och theorem},
  \href{http://arxiv.org/abs/1111.5546}{arXiv:1111.5546v1}, submitted for
  publication.

\bibitem{HS05}
M.~Hopkins and I.M. Singer, \emph{Quadratic functions in geometry, topology,and
  {M}-theory}, J. Diff. Geom. \textbf{70} (2005), 329--425.

\bibitem{K86}
M.~Karoubi, \emph{K-th$\acute{e}$orie multiplicative}, C. R. Acad. Sci. Paris
  S$\acute{\textrm{e}}$r. I Math \textbf{302} (1986), 321--324.

\bibitem{L94}
J.~Lott, \emph{$\mathbb{R}/\mathbb{Z}$ index theory}, Comm. Anal. Geom.
  \textbf{2} (1994), 279--311.

\bibitem{SS10a}
J.~Simons and D.~Sullivan, \emph{The {M}ayer-{V}ietoris property in
  differential cohomology}, arXiv:1010.5269.

\bibitem{SS08a}
\bysame, \emph{Axiomatic characterization of ordinary differential cohomology},
  J. Topol. \textbf{1} (2008), 45--56.

\bibitem{SS10}
\bysame, \emph{Structured vector bundles define differential {K}-theory},
  Quanta of maths, Clay Math. Proc., vol.~11, Amer. Math. Soc., Providence, RI,
  2010, pp.~579--599.

\end{thebibliography}
\end{document}